\documentclass[12pt]{article}
\usepackage{graphicx}
\usepackage{amsmath}
\usepackage{amsfonts}
\usepackage{amssymb}
\usepackage{url}
\newtheorem{theorem}{Theorem}[section]

\newtheorem{lemma}[theorem]{Lemma}

\newtheorem{proposition}[theorem]{Proposition}

\numberwithin{equation}{section}

\newenvironment{proof}[1][Proof]{\textbf{#1.} }{\ \rule{0.5em}{0.5em}}

\usepackage{color}
\usepackage[notref,notcite]{showkeys}

\topmargin=-0.5in \everymath{\displaystyle}
\begin{document}
\baselineskip=18pt

\pagenumbering{arabic}

\begin{center}
{\Large {\bf A  Vector-Host Epidemic Model with Spatial Structure and Age of Infection}}

\bigskip

W.E. Fitzgibbon and J.J. Morgan

Department of Mathematics

University of Houston\\
Houston, TX 77204, USA

\bigskip

Glenn F. Webb and Yixiang Wu

Department of Mathematics

Vanderbilt University\\
Nashville, TN 37212, USA

\bigskip

\vspace{0.2in}
\end{center}

\begin{abstract}
In this paper we study a diffusive age structured epidemic model with disease transmission between vector and host populations. The dynamics of the populations are described by  reaction-diffusion equations, with infection age structure of the host population incorporated to account for incubation periods. The disease is transmitted between vector and host populations in crisscross fashion. The existence of solutions of the model is studied by operator semigroup methods, and the asymptotic behavior of the solution is investigated.   
\end{abstract}

\medskip


\section{Introduction}
The objective of this paper is to analyze a spatial vector-host epidemic model. The model accounts for the random movement of the vector and host population in geographic regions, and the infection age-structure of the infected host population.
Many diseases are transmitted to human  by vectors, such as mosquito-borne diseases malaria, dengue, Zika and bug-borne Chagas. Such diseases are transmitted in a crisscross fashion: infected vectors transmit the disease to susceptible hosts, while susceptible vectors become infected through interactions with infected hosts.  Crisscross models for the circulation of diseases between vectors and hosts have been proposed and studied by many researchers in the past. For example, in \cite{Bailey1957, Dietz1988}  the authors studied the spread of malaria, and in \cite{Busenberg1988, inaba2004mathematical, Velasco1991} the authors studied the spread of Chagas disease by crisscross models. 

The vector and host populations are assumed to be confined in non-coincidental geographic regions. In particular we assume that the region of the vector population is contained in the region of the host population. The dispersal of individuals inside the regions is described by spatial diffusion terms with different diffusion rates for vector and host populations. We note that diffusion has been used to model the spartio-temporal spread of disease by a variety of authors \cite{Allen2008, Capasso1978, Fitzgibbon1994, Fitzgibbon2008, Webb1981}. 

After the susceptible hosts are infected with the disease, they are usually asymptomatic for a certain amount of time before becoming symptomatic and infectious. The incubation or non-infectious period of many  vector-host diseases is appreciable longer than the time it takes for an individual to travel from one place to another. Indeed  an outbreak in one locale could spread silently and globally via infected travelers, only be recognized days or even weeks later \cite{Knobler2006}. In order to incorporate the incubation period of the disease into the model, the host population is assumed to be structured by disease age. We note that the theory of age-structured population models has been well-developed recently, e.g. see \cite{Webb1985theory}. Epidemic models with diffusion and age-structure are studied in  \cite{Fitzgibbon1995, Fitzgibbon1996, Langlais1988, Webb1980}. For a review of diffusive age-structured models, we refer the readers to \cite{Webb2008population}.

Our goal is to understand how an infectious disease arises and spreads through vector-host populations in a geographical setting. Our model formulation assumes that diffusion descibes the movement of individuals, both vectors and hosts, within this geographical region. This assumption is an idealization, since the movement of both vectors and hosts, particularly hosts, may be extremely complex. We argue, however,  that the geographical spread of an epidemic, particularly from an initial small local outbreak, can be modeled by random diffusive processes. In this context diffusion indirectly models the average spatial spread of the underlying micro-biologic infectious agent (viral, bacterial, parasitic), rather then the local-time movement of hosts and vectors. The infectious agent  exists within the host and vector populations, and is not modeled directly. Reaction-diffusion mechanisms indirectly describe the way this infectious agent spreads in space and time within these populations.

Our paper is organized as follows: we propose the  vector-host model in the next section, which includes a system of reaction-diffusion equations for the vector population and a system with diffusion and age-structure for the host population; in section 3, we study the global existence of solution of the model using analytic semigroup to represent solutions of the diffusive age-structured equation; in section 4, we investigate the asymptotic behavior of the solution and prove that the solution always converges to the steady state. 

\section{The Model}
We assume that infected hosts are initially located in a small area of much larger host habitat. Essentially the infected hosts act as vectors introducing the disease to the region. This input corresponds to the disembarkation of infected travelers from a ship, plane or other means of conveyance. We also assume that the vector and host habitats are non-coincident with the vector habitat being a smaller sub-region of the larger host habitat. Recent works on the transmission of disease between species with non-coincident habitats include \cite{Fitzgibbon2004, Fitzgibbon2005, Anita2009}. A salient feature of our consideration will be a noninfectious period of asymptomatic incubation of the virus in host. The incubation period complicates any effort to prophylactically screen for the infection at points of embarkation and disembarkation. We assume that the virus has no deleterious effect on the vectors and that the vectors become immediately infective upon contact with infected infectious hosts with no period of incubation. The virus is assumed to be non-lethal and of relatively short duration in the host and for this reason demographic considerations for the hosts will not be included in the model.

We assume that our host population remain confined to a geographic region $\Omega\subset \mathbb{R}^2$. In particular we assume that $\Omega$ is a bounded domain in  $\mathbb{R}^2$ with smooth boundary $\partial\Omega$. The vector population is assumed to inhabit and remain confined to a bounded subdomain $\Omega_*\subset\Omega$, where $\partial\Omega_*$ is smooth with $\partial\Omega\cap\partial\Omega_*=\varnothing$. We assume that vector population disperse by means of Fickian diffusion with flux term $-d_1(x)\triangledown\rho$ and further require that $d_1(x)\ge d$ for a positive number $d_*$. The time-evolving spatial dependent density of the vector population is denoted by $\rho(x, t)$. The confinement of the vector population to $\Omega_*$ translates as the  Neumann boundary condition $d_1(x)\partial\rho/\partial\eta=0$ for $x\in\partial\Omega_*$, where $\eta$ denotes the unit outward normal on $\partial\Omega_*$. If $\beta(x)\ge \beta_*>0$ and $m(x)\ge m_*>0$ denote spatially dependent growth and logistic control coefficients respectively, the spatio-temporal evolution of the vector population is modeled by the diffusive logistic equation
\begin{equation}\label{rho}
\left\{
\begin{array} {lll}
   \frac{\partial \rho}{\partial t}-\triangledown\cdot d_1(x)\triangledown\rho=\beta(x)\rho-m(x)\rho^2,   &\ \ \ x\in\Omega_*,\ t> 0,\medskip\\
   \frac{\partial \rho}{\partial \eta}=0,  &\ \ \ x\in\partial\Omega_*,\ t> 0, \medskip\\
   \rho(x, 0)=\rho_0(x)>0,  &\ \ \ x\in\Omega_*.
\end{array}
\right.
\end{equation}

The following result appears in \cite{Fitzgibbon2008, Cantrell2004}
\begin{theorem}\label{theorem_rho}
Assume that the functions $\beta$ and  $m$ are  strictly positive continuous, $d_1$ is strictly positive continuously differentiable, and $\rho_0$ is nontrivial nonnegative continuous on $\bar\Omega_*$. Then there exists a unique classical solution of \eqref{rho} on $\Omega_*\times (0, \infty)$ such that 
$$0<\rho(x, t)\le \max\{\|\rho_0\|_{\Omega_*, \infty}, \|\beta\|_{\Omega_*, \infty}/m_*\}.$$ Moreover, there exists a unique positive classical  solution $\rho_*$ of
\begin{equation}\label{rho_str}
\left\{
\begin{array} {lll}
   -\triangledown\cdot d_1(x)\triangledown\rho=\beta(x)\rho-m(x)\rho^2,   &\ \ \ x\in\Omega_*,\medskip\\
   \frac{\partial \rho}{\partial \eta}=0,  &\ \ \ x\in\partial\Omega_*, \medskip
\end{array}
\right.
\end{equation}
such that
\begin{equation}
\lim_{t\rightarrow\infty} \|\rho(\cdot, t)-\rho_*\|_{\Omega_*, \infty}=0.
\end{equation}
\end{theorem}

Our model consists of four compartments:
\begin{itemize}
\item $\phi(x, t)$, $x\in\bar\Omega_*, t\ge 0$, denotes the time dependent spatial density of the uninfected vector population who have not contracted disease;
\item $\psi(x, t)$, $x\in\bar\Omega_*, t\ge 0$, dentoes the time dependent spatial density of the infected vector population who are infected with the disease and capable of transmitting it to the hosts;
\item $u(x, t)$, $x\in\bar\Omega, t\ge 0$, denotes the time dependent spatial density of the uninfected host population;
\item $i(x, a, t)$, $x\in\bar\Omega, a\ge0, t\ge 0$, denotes the time and age  dependent spatial density of  the infected host population.
\end{itemize}

Integrating $i(x, a, t)$  with respect to the age variable over $[0, \infty)$ gives $v(x, t)$, the time dependent spatial density of the infected host population:
\begin{equation}
v(x, t)=\int_0^\infty i(x, a, t) da.
\end{equation}

Infected hosts are assumed to go through an incubation period of length $\tau>0$ during which they are neither symptomatic nor infectious. After passing through the incubation period, the infected hosts become infectious. The time dependent spatial density of the infected and infectious host population is computed by
\begin{equation}
v_\tau(x, t)=\int_\tau^\infty i(x, a, t) da.
\end{equation}

 The two critical issues in understanding the transmission of the disease are the recruitment of infected vectors by means of direct contact with infected hosts and the recruitment of infected hosts by means of direct contact with infected vectors. The recruitment of infected vectors occurs via direct contact with infectious infected hosts, which is modeled by the incidence term
\begin{equation}
f_1(x, t, \phi(x, t), v_\tau(x, t))= \sigma_1(x)\phi(x, t)v_\tau(x, t).
\end{equation}

The virus is assumed to have no deleterious effect on the underlying demographics of vector population and we assume that there is no vertical transmission for of the virus among the host population. This considerations produce the following pair of reaction-diffusion type equations:
\begin{equation}\label{vector}
\left\{
\begin{array} {lll}
   \frac{\partial \phi}{\partial t}-\triangledown\cdot d_1(x)\triangledown\phi=\beta(x)\rho-\sigma_1(x)\phi v_\tau-m(x)\rho\phi,   &\ \ \ x\in\Omega_*,\ t> 0,\medskip\\
 \frac{\partial \psi}{\partial t}-\triangledown\cdot d_1(x)\triangledown\psi=\sigma_1(x)\phi v_\tau-m(x)\rho\psi,   &\ \ \ x\in\Omega_*,\ t> 0,\medskip\\
   \frac{\partial \phi}{\partial \eta}=\frac{\partial \psi}{\partial \eta}=0,  &\ \ \ x\in\partial\Omega_*,\ t> 0, \medskip\\
   \phi(x, 0)=\phi_0 ,  &\ \ \ x\in\Omega_*,\\
  \psi(x, 0)=0,  &\ \ \ x\in\Omega_*.
\end{array}
\right.
\end{equation}

With the introduction of the virus into the vectors, vector population is divided into two subclasses of population density, the susceptible class $\phi(x, t)$ and the infective class $\psi(x, t)$, where $\phi(x, t)+\psi(x, t)=\rho(x, t)$. Adding the two partial differential equations we obtain the diffusive logistic equation \eqref{rho} modeling the underlying vector demographics. We have assumed that $\psi(x, 0)=0$ consistent with our focus on the outbreak of the disease in a geographic region that was previously free of the virus. Mathematically we could assume any nonnegative initial data for $\psi(x, 0)$. 

The mechanism for host dispersion across the larger habitant $\Omega$ will also be Fickian diffusion with flux term $-\triangledown\cdot d_2(x)\triangledown u$ having strictly positive diffusivity $d_2(x)\ge d_*>0$ for all $x\in\bar\Omega_*$.  The evolution of $i(x, a, t)$ as we shall see will be governed by a diffusive age transport equation of the form
\begin{equation}
\frac{\partial i}{\partial t}+\frac{\partial i}{\partial a}-\triangledown\cdot d_2(x)\triangledown i=-\lambda(a) i, \ \ \ x\in\Omega, a\ge 0, t\ge0.
\end{equation}
We will subsequently specify an initial condition $i(x, a, 0)$, and an age boundary condition or birth function $i(x, 0, t)=B(x, t)$. The confinement of the host population to its habitant is prescribed by homogeneous Neumann boundary condition  on $\partial\Omega$. Recruitment of infected hosts  (represented by $B(x, t)$)  occurs as a result of the contact of susceptible hosts with infected vectors in $\Omega_*$. This is modeled by the incidence term
\begin{equation}
f_2(x, t, u(x, t), \psi(x, t))= \sigma_2(x)u(x, t)\psi(x, t),
\end{equation}
which becomes a loss term for the class of susceptible hosts.  Contact between hosts and vectors only occurs in the region inhabited by the vector population, and thus we assume $B(x, t)\equiv 0$, $x\in\Omega-\bar\Omega_*$ and make note of the fact that this assumption makes $B(x, t)$ discontinuous on $\Omega$. Recruitment into the infective class occurs at the age boundary $\alpha=0$ and we have
\begin{equation}\label{Birth}
i(x, 0, t)=B(x, t)=\left\{
\begin{array} {lll}
   f_2(x, t, u(x, t), \psi(x, t))=\sigma_2(x)u(x, t)\psi(x, t)   &\ \ \ x\in\bar\Omega_*,\ t> 0,\medskip\\
  0,  &\ \ \ x\in\Omega-\bar\Omega_*, t>0.
\end{array}
\right.
\end{equation}
The virus is assumed to be non-fatal to the host population and removed from the infective class is due to recovery. We also assume that the recovered hosts gain permanent immunity and hence are removed from the ongoing dynamics of the system.  The recovery rate is assumed to be vary with the age of infection and given by $\lambda(a)$ with $\lambda(a)\ge \lambda_*>0$ for all $a\in [0, \infty)$. We remark that it would be natural to assume that $\lambda(a)$ eventually becomes sharply increasing in $a$. 

We model the introduction of infected hosts into the region by specifying an age dependent spatial density for $i(x, a, 0)$. Here we envision the introduction of an extremely small number of infected hosts distributed over an extremely small subarea $\Omega_{**}$ of $\Omega$, i.e. $|\Omega_{**}|<<|\Omega|$ and $\int_{\Omega_{**}}\int_0^\infty i(x, a, 0)dadx<<\int_\Omega u(x, 0)dx$. The infected host population is assumed to have age density $z_0(a)$ initially distributed over $\Omega_{**}$. The distribution is modeled by a probability density function, $k(x)$, defined on $\Omega$, which vanishes identically outside of $\Omega_{**}$. The spatial density of initial infected population is $D(x)=k(x)\int_0^\infty z_0(a)da$. These considerations lead to the following partial differential equations which describe the temporal and spatial circulation of the virus in the host population:
\begin{equation}\label{host}
\left\{
\begin{array} {lll}
   \frac{\partial u}{\partial t}-\triangledown\cdot d_2(x)\triangledown u=-B(x, t),   &\ \ \ x\in\Omega,\ t> 0,\medskip\\
 \frac{\partial i}{\partial t}+\frac{\partial i}{\partial a}-\triangledown\cdot d_2(x)\triangledown i=-\lambda(a) i,   &\ \ \ x\in\Omega, \ a>0,\ t> 0,\medskip\\
i(x, 0, t)=B(x, t),      &\ \ \ x\in\Omega, \ t> 0,\medskip\\
\frac{\partial u}{\partial \eta}=\frac{\partial i}{\partial \eta}=0,  &\ \ \ x\in\partial\Omega,\ a>0, \ t> 0, \medskip\\
   i(a, x, 0)=i_0(a, x)=z_0(a)k(x),  &\ \ \ x\in\Omega,\ a>0,\\
  u(x, 0)= u_0(x),  &\ \ \ x\in\Omega.
\end{array}
\right.
\end{equation}

We make the following assumptions:
\begin{enumerate}
\item[A0.] $d_1\in C^1(\bar\Omega_*)$ with $d_1(x)\ge d_*$ for all $x\in\bar\Omega_*$; $d_2\in C^1(\bar\Omega)$ with $d_2(x)\ge d_*$ for all $x\in\bar\Omega$;
\item[A1.] $z_0\in C^1(\mathbb{R}^+)\cap L_\infty(\mathbb{R}^+)\cap L_1(\mathbb{R}^+)$ is nontrivial nonnegative with  $z_0(0)=0$;
\item[A2.]  $k$ is nonnegative continuous function on $\bar\Omega$ such that $k=0$ on $\Omega-\Omega_{**}$ and
\begin{equation*}
\int_{\Omega_{**}} k(x) dx=1;
\end{equation*}
\item[A3.] $\sigma_1$ and $\sigma_2$ are strictly positive continuous functions on  $\bar\Omega_*$;
\item[A4.] $m$ and $\beta$ are bounded continuous functions on $\bar\Omega$ with $m(x)\ge m_*>0$ and $\beta(x)\ge \beta_*>0$ for all $x\in\bar\Omega$;
\item[A5.] $\lambda\in C^1(\mathbb{R}^+)$ with $\lambda(a)\ge \lambda_*>0$ for all $a\ge 0$;
\item[A6.] $u_0\in C(\bar\Omega)$ with $u_0(x)\ge (\not\equiv) 0$ for $x\in\bar\Omega$ and $\phi_0\in C(\bar\Omega_*)$ with $\phi_0(x)\ge (\not\equiv) 0$ for all $x\in\bar\Omega_*$.
\end{enumerate}
%

\section{Existence of solutions}
We will use semigroup theory to represent solutions of the diffusive age transport equation, e.g. see \cite{Pazy2012}. We let $T(t), t\ge 0$, be the analytic semigroup on $C(\bar\Omega)$ with infinitesimal generator $A$, which is defined by
\begin{equation}
(Aw)(x)=\triangledown\cdot d_2(x)\triangledown w(x), \ \ w\in D(A) \text{ and } x\in\bar\Omega ,
\end{equation}
with
\begin{equation}
D(A)=\left\{w\in C^2(\bar\Omega): \frac{\partial w}{\partial \eta}=0 \ on \ \partial\Omega\right\}.
\end{equation}

Let $c\in\mathbb{R}$ and $t_0\ge 0$ such that $t_0+c>0$. If $w_0\in C(\bar\Omega)$, then the classical solutions of
 \begin{equation}\label{w}
\left\{
\begin{array} {lll}
   \frac{\partial w}{\partial t}-\triangledown\cdot d_2(x)\triangledown w=-\lambda(t+c)w,   &\ \ \ x\in\Omega,\ t>t_0,\medskip\\
   \frac{\partial w}{\partial \eta}=0,  &\ \ \ x\in\partial\Omega,\ t> t_0, \medskip\\
   w(x, t_0)=w_0(x),  &\ \ \ x\in\Omega.
\end{array}
\right.
\end{equation}
has 
representation
\begin{equation}
w(\cdot, t)= e^{-\int_{t_0}^t\lambda(s+c)ds}T(t-t_0)w_0.
\end{equation}
The maximum principle guarantees that $$\|w(\cdot, t)\|_{\Omega, \infty}\le e^{-\lambda_*(t-t_0)}\|w_0\|_{\Omega, \infty} \ \ \text{ for } t\ge t_0.$$ 
%

We are now in a position to establish our well-posedness result. As previously observed in \cite{Fitzgibbon19952}, the existence of an incubation period $[0, \tau]$ allows us to effectively decouple the nonlinearity and obtain existence results via linear theory.
\begin{theorem}
If asumptions A1-A6  are satisfied, then there exists unique coupled positive solution pairs $\{\phi(x, t), \psi(x, t)\}$ and $\{u(x, t), i(x, a, t)\}$ of  continuous functions  on $\Omega_*\times (0, \infty)$ and $\Omega\times (0, \infty)\times (0, \infty)$ respectively, which satisfy the system \eqref{vector} and \eqref{host}.
\end{theorem}
\begin{proof}
Adding the equations for the susceptible and infected hosts produces the diffusive logistic equation
\begin{equation}\label{rho1}
\left\{
\begin{array} {lll}
   \frac{\partial \rho}{\partial t}-\triangledown\cdot d_1(x)\triangledown\rho=\beta(x)\rho-m(x)\rho^2,   &\ \ \ x\in\Omega_*,\ t\in (0, T],\medskip\\
   \frac{\partial \rho}{\partial \eta}=0,  &\ \ \ x\in\partial\Omega_*,\ t\in (0, T], \medskip\\
   \rho(x, 0)=\rho_0(x)=\phi_0\ge (\not\equiv)0,  &\ \ \ x\in\Omega_*.
\end{array}
\right.
\end{equation}
Theorem \ref{theorem_rho} guarantees classical uniformly bounded positive solution and allows us to assume that $\rho(x, t)$ is a known quantity. We initiate a method of steps argument and look for solutions on the time interval $[0, \tau]$ and use a characteristic argument to find a representation of the diffusive age transport equation. Adapting arguments appearing in \cite{Fitzgibbon19952}, we let $c\in\mathbb{R}$ and define the cohort function $w_c(x, t)$. Direct computation yields
\begin{equation}
\frac{\partial w_c(x, t)}{\partial t}=\mathcal{L}\  i(x, t+c, t),
\end{equation}
where $\mathcal{L}$ denotes the formal diffusive age transport operator $\mathcal{L}=\partial/\partial t+\partial/\partial a-\triangledown\cdot d_2(x)\triangledown$. We now introduce $t_c=\max\{0, -c\}$ and observe that the solution of 
\begin{equation*}
\left\{
\begin{array} {lll}
   \frac{\partial w_c}{\partial t}+\frac{\partial w_c}{\partial a}-\triangledown\cdot d_2(x)\triangledown w_c=-\lambda(t+c) w_c,   &\ \ \ x\in\Omega,\ t> t_c,\medskip\\
   \frac{\partial w_c}{\partial \eta}=0,  &\ \ \ x\in\partial\Omega,\ t> t_c.
\end{array}
\right.
\end{equation*}
is
\begin{equation*}
w_c(\cdot, t)= e^{-\int_{t_c}^t\lambda(s+c)ds}T(t-t_c)w_c(\cdot, t_c).
\end{equation*}
If $c=a-t\ge 0$, then $t_c=0$, and we have for $a\ge t$
\begin{eqnarray*}
i(\cdot, a, t)=i(\cdot, a-t+t, t)&=&w_c(\cdot, t)\\
   &=&e^{-\int_{0}^t\lambda(s+c)ds}T(t)w_c(\cdot, 0)\\
  &=&e^{-\int_{0}^t\lambda(s+a-t)ds}T(t)i(\cdot, a-t, 0)\\
  &=&e^{-\int_{0}^t\lambda(s+a-t)ds}T(t)z_0(a-t)k(\cdot).
\end{eqnarray*}
If $c=a-t<0$, then $t_c=t-a$, and we have for $t> a$
\begin{eqnarray*}
i(\cdot, a, t)&=&w_c(\cdot, t)\\
   &=&e^{-\int_{t-a}^t\lambda(s+c)ds}T(a)w_c(\cdot, t-a)\\
  &=&e^{-\int_{t-a}^t\lambda(s+a-t)ds}T(a)i(\cdot, 0, t-a)\\
   &=&e^{-\int_{0}^a\lambda(s)ds}T(a)B(\cdot, t-a).
\end{eqnarray*}
Hence we have
\begin{equation}\label{irep}
i(x, a, t)=\left\{
\begin{array}{lll}
e^{-\int_{0}^t\lambda(s+a-t)ds}T(t)z_0(a-t)k(x),  \ \ \ &x\in\Omega, t\le a, \\
e^{-\int_{0}^a\lambda(s)ds}T(a)B(x, t-a), \ \ \ &x\in\Omega, t>a.
\end{array}
\right.
\end{equation}

 The infected become infective for $a \ge \tau$. Thus if $t\in [0, \tau]$, we can determine the time dependent spatial density of the infected infective population
\begin{equation*}
v_\tau(x, t)=\int_\tau^\infty i(x, a, t)da=\int_\tau^\infty e^{-\int_0^t\lambda(s+a-t)ds}T(t)z_0(a-t)k(x)da.
\end{equation*}
We observe that the assumptions on $z_0(a)$ and $k(x)$ insures that $v_\tau(x, t)> 0$ for $x\in\Omega$ and $t\in(0, \tau]$. Since $v_\tau(x, t)$ has been calculated and we have seen that we are guaranteed a solution to the diffusive logistic equationm, we can view \eqref{vector} as a coupled linear parabolic system. Standard parabolic theory guarantees a classical solution on $\Omega_*\times [0, \tau]$. Maximum principle arguments now insure that the solution $(\phi(x, t)), \psi(x, t)$ is positive on $\bar\Omega_*\times (0, \tau]$. Since the existence of $\psi(x, t)$ is now assured, the partial differential equation describing the depletion of the susceptible class,
\begin{equation*}
\frac{\partial u}{\partial t}-\triangledown\cdot d_2(x)\triangledown u=-B(x, t)=\left\{
\begin{array}{lll}
 -\sigma_2(x)u(x, t)\psi(x, t), \ \ \ &x\in\Omega_*, t\in(0, \tau],\\
0 \ \ \ &x\in\Omega-\bar\Omega_*, t\in(0, \tau]
\end{array}
\right.
\end{equation*}
can be viewed as linear and therefore there exists a positive  solutions on $\bar\Omega\times (0, \tau]$. Knowing $u(x, t)$ and $\psi(x, t)$ permits calculation of the birth function that provides the mechanism for entry into the infective class at the age boundary $a=0$
\begin{equation*}
i(x, 0, t)=B(x, t)=\left\{
\begin{array}{lll}
 \sigma_2(x)u(x, t)\psi(x, t), \ \ \ &x\in\bar\Omega_*, t\in(0, \tau],\\
0 \ \ \ &x\in\Omega-\bar\Omega_*, t\in(0, \tau].
\end{array}
\right.
\end{equation*}
We complete the cycle of this argument with the observation that our knowledge of $B(x, t)$ and $i_0(x, a)=z_0(a)k(x)$ facilitates direct computation of the solution $i(x, a, t)$ to the diffusive age transport equation on $\bar\Omega\times (0, \infty)\times (0, \tau]$ by \eqref{irep} . By our assumption $z_0(0)=0$ and $\lambda(a)\ge \lambda_*>0$,  $i(x, a, t)$ is uniformly continuous on $\Omega\times (0, \infty)\times (0, \tau]$. By the positivity of the semigroup $T(t)$, $i(x, a, t)$ is positive on $\Omega\times (0, \infty)\times (0, \tau]$.

We then look for a solution for $t\in[\tau, 2\tau]$. Again we find the spatial density of the infective hosts by integrating $i(x, a, t)$ with respect to $a$ on $[\tau, \infty)$. However in this case $t\ge\tau$, we have
\begin{equation*}
v_\tau(x, t)=\int_\tau^\infty i(x, a, t)da=\int_\tau^t i(x, a, t)da+\int_t^\infty i(x, a, t)da.
\end{equation*}
Each of these integrals on the right hand side can be evaluated using previous knowledge. In the case of the second integral, we observe that
\begin{equation}\label{in1}
\int_t^\infty i(x, a, t)da=\int_t^\infty e^{-\int_0^t\lambda(s+a-t)ds}T(t)z_0(a-t)k(x)da.
\end{equation}
In the case of the first integral, we have $\tau\le a<t\le 2\tau$, which implies $t-a<\tau$. Hence
\begin{equation}\label{in2}
\int_\tau^t i(x, a, t)da=\int_\tau^t e^{-\int_0^a\lambda(s)ds}T(a)B(x, t-a)da.
\end{equation}
In the last integral, 
\begin{equation*}
B(x, t-a)=\left\{
\begin{array}{lll}
\sigma_2(x)u(x,t-a)\psi(x, t-a), \ \ \ &t-a \in[0, \tau], x\in\bar\Omega_*,\\
0 , \ \ \ &t-a \in[0, \tau], x\in\Omega-\bar\Omega_*,
\end{array}
\right.
\end{equation*}
which has been obtained from the previous step.
Thus, we can determine both integrals \eqref{in1}-\eqref{in2} that define $v_\tau(x, t)$ for $t\in [\tau, 2\tau]$. We can insert the pre-determined functions $\rho(x, t)$ and $v_\tau(x, t)$ into the system describing the evolution of $\psi$ and $\phi$, and again obtain a linear system. Hence we get a positive  solution on $[\tau, 2\tau]$. Now since $\psi(x, t)$ is determined for $t\in[\tau,2\tau]$, we can reduce the equation for $u$ to a linear equation for which we can readily obtain a positive solution. Knowing $u(x, t)$ and $\psi(x, \tau)$ on $[\tau, 2\tau]$ permits calculation of $B(x, t)$ and hence $i(x, a, t)$ on $[\tau, 2\tau]$. Our analysis thus so far shows that we can determine the solution $(\phi(x, t), \psi(x, t))$ and $(u(x, t), i(x, a, t))$ of positive uniformly continuous functions  on $\Omega_*\times [\tau, 2\tau]$ and $\Omega\times [0, \infty)\times [\tau, 2\tau]$ respectively. It is
evident that the proceeding argument in this manner step our way across $t\in [0, \infty)$, and guarantee the unique positive solution $(\phi(x, t), \psi(x, t))$ and $(u(x, t), i(x, a, t))$ on $\Omega_*\times [0, \infty]$ and $\Omega\times [0, \infty)\times [0, \infty)$ respectively.
\end{proof}

\section{Asymptotic behavior}
Analysis of the long term behavior of solutions will require us the following a priori bounds.
\begin{proposition}\label{prop_bound}
Let A1-A6 hold and let $(\phi(x, t), \psi(x, t), u(x, t), i(x, a, t))$ be the solution of system \eqref{vector} and \eqref{host}. Then there exists  $M>0$ such that the following hold:
\begin{eqnarray*}
&&\sup\{\|\phi(\cdot, t)\|_{\Omega_*, \infty},    \|\psi(\cdot, t)\|_{\Omega_*, \infty}, \|u(\cdot, t)\|_{\Omega, \infty}, \|v(\cdot, t)\|_{\Omega, \infty}, \|v_\tau(\cdot, t)\|_{\Omega, \infty}: \ t> 0\}<M,\\
&&\sup\{\|i(\cdot, a, t)\|_{\Omega, \infty}: \ a\ge0, t>0\}<M.
\end{eqnarray*}
For any $p>1$, there exists $M_p>0$ such that
\begin{eqnarray*}
&&\sup\{\|\partial\phi(\cdot, t)/\partial t\|_{\Omega_*, p},    \|\partial\psi(\cdot, t)/\partial t\|_{\Omega_*, p}, \|\partial u(\cdot, t)/\partial t\|_{\Omega, p}, \|\partial v(\cdot, t)/\partial t\|_{\Omega, p}: \ t> 0\}<M_p,\\
&&\sup\{\|\triangledown\phi(\cdot, t)\|_{\Omega_*, p},    \|\triangledown\psi(\cdot, t)\|_{\Omega_*, p}, \|\triangledown u(\cdot, t)\|_{\Omega, p}, \|\triangledown v(\cdot, t)\|_{\Omega, p}: \ t> 0\}<M_p,\\
&&\sup\{\|\triangledown^2\phi(\cdot, t)\|_{\Omega_*, p},    \|\triangledown^2\psi(\cdot, t)\|_{\Omega_*, p}, \|\triangledown^2 u(\cdot, t)\|_{\Omega, p}, \|\triangledown^2 v(\cdot, t)\|_{\Omega, p}: \ t> 0\}<M_p.
\end{eqnarray*}
\end{proposition}
\begin{proof}
By the comparison principle, we have $$0\le \rho(x, t)\le \max\{\|\beta\|_{\Omega_*, \infty}/m_*, \|\rho_0\|\}.$$ This together with the non-negativity of $\phi(x, t)$ and $\psi(x, t)$ implies that
\begin{equation*}
\sup\{\|\phi(\cdot, t)\|_{\Omega_*, \infty},    \|\psi(\cdot, t)\|_{\Omega_*, \infty}: \ t\ge 0\}\le \max\{\|\beta\|_{\Omega_*, \infty}/m_*, \|\phi_0\|_{\Omega_*, \infty}\}.
\end{equation*}
We have established that $u(x, t)$ is nonnegative. Therefore, $\partial u/\partial t-\triangledown\cdot d_2(x)\triangledown u\le 0$ and the maximum principle implies
\begin{equation*}
\|u(\cdot, t)\|_{\Omega, \infty} \le \|u_0\|_{\Omega, \infty}.
\end{equation*}
We can observe that
\begin{equation*}
\|B(\cdot, t)\|_{\Omega, \infty} \le \|\sigma_2\|_{\Omega_*, \infty}\|u_0\|_{\Omega, \infty}\max\{ \|\beta\|_{\Omega_*, \infty}/m_*, \|\phi_0\|_{\Omega_*, \infty}\}\equiv N.
\end{equation*}
We now use ther representation of $i(x, a, t)$ to observe that
\begin{equation*}
\|i(\cdot, a, t)\|_{\Omega, \infty}\le \left\{
\begin{array} {lll}
  \|k\|_{\Omega, \infty}\|z_0\|_{[0, \infty), \infty},   &\ \ \ t\le a,\medskip\\
   N,  &\ \ \ t>a.
\end{array}
\right.
\end{equation*}

If we integrate the age transport equation on $[0, \infty)$, we get
\begin{equation*}
\frac{\partial v}{\partial t}-\triangledown\cdot d_2(x)\triangledown v= B(x, t)-\int_0^\infty\lambda(a)i(\cdot, a, t)da, \ \ \ x\in\Omega, t\ge0.
\end{equation*}
Let $w(x,t)=u(x,t)+v(x,t)$. We observe that
\begin{equation*}
\frac{\partial w}{\partial t}-\triangledown\cdot d_2(x)\triangledown w\le0.
\end{equation*}
Invoking the maximum principle, we have
\begin{equation*}
\|v(\cdot, t)\|_{\Omega, \infty}\le \|w(\cdot, t)\|_{\Omega, \infty}\le \|w(\cdot, 0)\|_{\Omega, \infty}\le\|k\|_{\Omega, \infty}\|z_0\|_{[0, \infty), \infty}.
\end{equation*}
By $\int_\tau^\infty i(x, a, t) da\le\int_0^\infty i(x, a, t)da$, we see that $\|v_\tau(\cdot, t)\|_{\Omega, \infty}\le \|v(\cdot, t)\|_{\Omega, \infty}$ and  we have obtained a uniform a priori bound for $\|\phi(\cdot, t)\|_{\Omega_*, \infty}$, $\|\psi(\cdot, t)\|_{\Omega_*, \infty}$, $\|u(\cdot, t)\|_{\Omega, \infty}$, $\|v(\cdot, t)\|_{\Omega, \infty}$, and $\|v_\tau(\cdot, t)\|_{\Omega, \infty}$. The uniform estimate on the spatial derivatives follows from a common semigroup calculation in the fractional power spaces, and the estimate for the time derivative just follows from \eqref{vector}, \eqref{host}, and the other estimates.
\end{proof}

We are now in position to provide a complete description of the asymptotic behavior of the solution quadruple $(\phi(x, t), \psi(x, t), u(x, t), i(x, a, t))$.
\begin{lemma}\label{lemma_v}
Suppose that assumptions A1-A6 hold, and let $(\phi(x, t), \psi(x, t), u(x, t), i(x, a, t))$ be the solution of system \eqref{vector} and \eqref{host}. Then we have
\begin{equation*}
\lim_{t\rightarrow\infty} \|v(\cdot, t)\|_{\Omega, 1}=0
\end{equation*}
\end{lemma}
\begin{proof}
Setting $w(x, t)=u(x, t)+v(x, t)$ and adding the differential equations for $u(x, t)$ and $v(x, t)$, we find
\begin{equation*}
\frac{\partial w}{\partial t}-\triangledown\cdot d_2(x)\triangledown w+\int_0^\infty \lambda(a)i(x, a, t)da=0, \ \ \ x\in\Omega, t\ge0.
\end{equation*}
Integrating both sides of the equation over $\Omega\times(0, t)$ and noticing that $\lambda(a)\ge \lambda_*>0$, we obtain
\begin{equation*}
\|w(\cdot, t)\|_{\Omega, 1}+\lambda_*\int_0^t\|v(\cdot, s)\|_{\Omega, 1} ds\le \|w(\cdot, 0)\|_{\Omega, 1}.
\end{equation*}
Hence,
\begin{equation}\label{vbound}
\int_0^\infty \|v(\cdot, s)\|_{\Omega, 1} ds     \le \|w(\cdot, 0)\|_{\Omega, 1}.
\end{equation}
This fact together with the uniform a priori bound on $\|\partial v/\partial t\|_{\Omega, 1}$ guaranteed by Proposition \ref{prop_bound}  insure that $$\lim_{t\rightarrow\infty} \|v(\cdot, t)\|_{\Omega, 1}=0.$$
\end{proof}

\begin{lemma}\label{lemma_psi}
Suppose that assumptions A1-A6 hold, and let $(\phi(x, t), \psi(x, t), u(x, t), i(x, a, t))$ be the solution of system \eqref{vector} and \eqref{host}. Then we have
\begin{equation*}
\lim_{t\rightarrow\infty} \|\psi\cdot, t)\|_{\Omega_*, 2}=0
\end{equation*}
\end{lemma}
\begin{proof}
If we integrate the differential equation for $\psi$ over $\Omega_*$ and followed by the integration with respect to $t$, we have
\begin{equation*}
\int_{\Omega_*} \psi(x, t) dx = \int_0^t\int_{\Omega_*} \sigma_1(x)\phi(x,s)v_\tau(x,s)dx-\int_0^t\int_{\Omega_*}m(x)\rho(x,s)\psi(x,s)dxds.
\end{equation*}
Recalling $0\le\psi(x, t)\le \psi(x,t)+\phi(x,t)=\rho(x,t)$ and $m(x)\ge m_*>0$, we obtain
\begin{equation*}
\|\psi(\cdot, t)\|_{\Omega_*, 1}+m_*\int_0^t \|\psi(\cdot, s)\|^2_{\Omega_*, 2}ds \le \|\sigma_1\|_{\Omega_*, \infty}\|\phi\|_{\Omega_*, \infty}\int_0^t\|v(\cdot, s)\|_{\Omega, 1}ds.
\end{equation*}
So by \eqref{vbound} and Proposition \ref{prop_bound}, we have
\begin{equation}\label{psibound}
\int_0^\infty \|\psi(\cdot, s)\|^2_{\Omega_*, 2} ds  <\infty.
\end{equation}
 Multiplying both sides of the equation for $\psi$ by $\psi$ and integrating it over $\Omega_*$, we obtain
\begin{equation}\label{psi}
\frac{1}{2}\frac{d}{dt} \|\psi(\cdot, t)\|^2_{\Omega_*, 2} +\int_\Omega d_2(x)|\triangledown \psi(\cdot, t)|^2dx =  \int_{\Omega_*}\sigma_1\phi v_\tau vdx-\int_{\Omega_*}m\rho\psi^2dx.
\end{equation}
By Proposition \ref{prop_bound}, $d\|\psi(\cdot, t)\|_{\Omega_*, 2}^2/dt$ is uniformly bounded for $t>0$, and this together with \eqref{psibound} implies that  
$$\lim_{t\rightarrow\infty}\|\psi(\cdot, t)\|_{\Omega_*, 2}=0.$$
\end{proof}

\begin{lemma}\label{lemma_u}
Suppose that assumptions A1-A6 hold, and let $(\phi(x, t), \psi(x, t), u(x, t), i(x, a, t))$ be the solution of system \eqref{vector} and \eqref{host}. Then there exists a constant $u_*\ge 0$ such that
\begin{equation}\label{ueql}
\lim_{t\rightarrow\infty} \|u(\cdot, t)-u_*\|_{\Omega, 2}=0.
\end{equation}
\end{lemma}
\begin{proof}
Let $U(t)=\int_\Omega u(x, t)dx$ and $\bar u(t)=U(t)/|\Omega|$. Integrating both sides of 
\begin{equation}\label{ueq}
\frac{\partial u}{\partial t}-\triangledown\cdot d_2(x)\triangledown u=-B(x, t)
\end{equation} 
over $\Omega$, we get 
$$\frac{\partial }{\partial t}U(t)=- B(x, t)\le 0.$$
So $U(t)$ is decreasing and there exists $u_*\ge 0$ such that 
$$\lim_{t\rightarrow \infty}\bar u(t)=u_*.$$ 

By the Poincare inequality, there exists $C>0$ such that for all $t> 0$
\begin{equation}\label{poincare}
 \|u(\cdot, t)-\bar u(t)\|_{\Omega_*, 2}\le  C\|\triangledown u(\cdot, t)\|_{\Omega_*, 2}.
\end{equation}
Noticing \eqref{ueq} and \eqref{poincare}, we compute
\begin{equation*}
\begin{array} {lll}
\frac{1}{2}\frac{d}{dt} \int_\Omega (u-\bar u)^2dx&=\int_\Omega (u-\bar u)(\triangledown\cdot d_2\triangledown u- B(x, t))dx\\
&=-\int_\Omega d_2 |\triangledown u|^2 dx-\int_\Omega (u-\bar u)B(x, t) dx\\
&\le -d_*C\int_\Omega (u-\bar u)^2dx+\frac{1}{2}d_*C\int_\Omega (u-\bar u)^2dx +\frac{ K}{2}\int_{\Omega_*}|\psi(\cdot, t)|^2dx
\end{array}
\end{equation*}
for some $K>0$, where we used Cauchy's inequality in the last step. Then by the Gronwall's inequality, we have
\begin{equation*}
 \|u(\cdot, t)-\bar u(t)\|^2_{\Omega_*, 2}\le \|u_0-\bar u_0\|^2_{\Omega_*, 2} K e^{-d_*Ct} e^{\int_0^t e^{d_*Cs}\|\psi(\cdot, s)\|^2_{\Omega_*, 2}ds }.
\end{equation*}
By Lemma \ref{lemma_psi}, we have
\begin{equation*}
\lim_{t\rightarrow\infty} e^{-d_*Ct} e^{\int_0^t e^{d_*Cs}\|\psi(\cdot, s)\|^2_{\Omega_*, 2}ds } = \lim_{t\rightarrow \infty} \frac{e^{d_*Ct} \|\psi(\cdot, t)\|^2_{\Omega_*, 2}}{ d_*C e^{d_*Ct}}=0.
\end{equation*}
It then follows that 
\begin{equation*}
\lim_{t\rightarrow\infty} \|u(\cdot, t)-\bar u(t)\|_{\Omega, 2}=0.
\end{equation*}
So we have
\begin{equation*}
\lim_{t\rightarrow\infty} \|u(\cdot, t)-u_*\|_{\Omega, 2}\le  \lim_{t\rightarrow\infty} \|u(\cdot, t)-\bar u(t)\|_{\Omega, 2} + \lim_{t\rightarrow\infty} \|\bar u(t)-u_*\|_{\Omega, 2}=0.
\end{equation*}
\end{proof}

\begin{theorem}\label{theorem_con}
Let A1-A6 hold and let $(\phi(x, t), \psi(x, t), u(x, t), i(x, a, t))$ be the solution of system \eqref{vector} and \eqref{host}. Then we have
\begin{equation*}
\lim_{t\rightarrow\infty} \|v(\cdot, t)\|_{\Omega, \infty}=0, \ \text{ and }\ \lim_{t\rightarrow\infty} \|\psi(\cdot, t)\|_{\Omega, \infty}=0,
\end{equation*}
and there exists a constant $u_*>0$ such that
\begin{equation*}
\lim_{t\rightarrow\infty} \|u(\cdot, t)-u_*\|_{\Omega, \infty}=0;
\end{equation*}
Moreover,
\begin{equation}\label{con_phi}
\lim_{t\rightarrow\infty} \|\phi(\cdot, t)-\rho_*\|_{\Omega_*, \infty}=0,
\end{equation}
where $\rho_*$ is the unique positive solution of
\begin{equation*}
\left\{
\begin{array} {lll}
   -\triangledown\cdot d_1(x)\triangledown\rho=\beta(x)\rho-m(x)\rho^2,   &\ \ \ x\in\Omega_*,\medskip\\
   \frac{\partial \rho}{\partial \eta}=0,  &\ \ \ x\in\partial\Omega_*. \medskip
\end{array}
\right.
\end{equation*}
\end{theorem}
\begin{proof} 
By the Sobolev imbedding theorem, we have that the imbeddings $W^{1, p}(\Omega)\subseteq C(\bar\Omega)$ and $W^{1, p}(\Omega_*)\subseteq C(\bar\Omega_*)$ are compact for $p>n/2$. So by Proposition \ref{prop_bound}, the orbits $\{\phi(\cdot, t), t\ge 1\}$ and $\{\psi(\cdot, t), t\ge 1\}$ are precompact in $C(\bar\Omega_*)$, and $\{u(\cdot, t), t\ge 1\}$ and $\{v(\cdot, t), t\ge 1\}$ are precompact in $C(\bar\Omega)$. Then the uniform convergence of $v, \psi$ and $u$ just follows from Lemmas \ref{lemma_v}-\ref{lemma_u}.

Theorem \ref{theorem_rho} states that  $$\lim_{t\rightarrow\infty} \|\rho(\cdot, t)-\rho_*\|_{\Omega_*, \infty}=0.$$ Then \eqref{con_phi} follows from the uniform convergence of $\psi$ to zero.

To complete the proof, we still need to show $u_*>0$, and this will be done in the following two lemmas.
\end{proof}

\begin{lemma}\label{lemma_psibound}
Suppose that assumptions A1-A6 hold, and let $(\phi(x, t), \psi(x, t), u(x, t), i(x, a, t))$ be the solution of system \eqref{vector} and \eqref{host}. If $u_*=0$ in Theorem \ref{theorem_con}, then
\begin{equation}\label{psi_uni}
 \int_0^\infty \|\psi(\cdot, s)\|_{\Omega_*, \infty} ds<\infty.
\end{equation}
\end{lemma}
\begin{proof}
By the assumption $u_*=0$, for a given $\epsilon>0$ (to be specified later) we may assume without loss of generality that $\|u(\cdot, t)\|_{\Omega, \infty}<\epsilon$ for all $t\ge 0$. Since $\rho(\cdot, t)\rightarrow \rho_*>0$  in $C(\bar\Omega_*)$ as $t\rightarrow \infty$ and $m(x)\ge m_*>0$, we may assume without loss of generality that $m(x)\rho(x, t)\ge \lambda_1$ with some positive number $\lambda_1$ for all $x\in\bar\Omega_*$ and $t\ge 0$. As for convenience, we choose $\lambda_1$ small such that $\lambda_*>\lambda_1$ where $\lambda_*$ is in assumption A5. 

  Let $A_1$ be an operator in $C(\bar \Omega_*)$ defined as 
$$A_1 w=\triangledown\cdot d_1\triangledown w-\lambda_1 w, \ \ \ w\in D(A_1),$$
$$D(A_1)=\{w\in C(\bar\Omega_*): \ \ w\in C^2(\bar\Omega_*) \ \text{ and } \ \frac{\partial w}{\partial \eta}=0  \text{ on } \partial\Omega_*\}.$$

  Let $A_2$ be an operator in $C(\bar \Omega)$ defined as 
$$A_2 w=\triangledown\cdot d_2\triangledown w-\lambda_* w, \ \ \ w\in D(A_2),$$
$$D(A_2)=\{w\in C(\bar\Omega): \ \ w\in C^2(\bar\Omega) \ \text{ and } \ \frac{\partial w}{\partial \eta}=0  \text{ on } \partial\Omega\}.$$
Let $\{T_1(t): t\ge 0\}$ be the semigroup generated by $A_1$ in $C(\bar\Omega_*)$ and $\{T_2(t): t\ge 0\}$ be the semigroup generated by $A_2$ in $C(\bar\Omega)$. There exists $M_1>0$ such that 
$$\|T_1(t)\|\le M_1e^{-\lambda_1 t} \ \ \ \text{ and } \  \ \ \|T_2(t)\|\le M_1e^{-\lambda_* t}.$$

By the second equation of \eqref{vector}, we have 
\begin{eqnarray*}
\psi(\cdot, t) &=&\int_0^t T_1(t-s)(\sigma_1\phi(\cdot, s)v_\tau(\cdot, s)- (m\rho(\cdot, s)-\lambda_1)\psi(\cdot, s))ds\\
          &\le & \int_0^t T_1(t-s)(\sigma_1\phi(\cdot, s)v(\cdot, s))ds.
\end{eqnarray*}
It then follows that 
\begin{eqnarray}
\|\psi(\cdot, t)\|_{\Omega_*, \infty}         &\le & \int_0^t \| T_1(t-s)(\sigma_1\phi(\cdot, s)v(\cdot, s)) \|_{\Omega_*, \infty}ds \nonumber\\
              &\le& M_2 \int_0^t e^{-\lambda_1(t-s)}\|v(\cdot, s)\|_{\Omega, \infty}ds,\label{conv1}
\end{eqnarray}
where $M_2=M_1M \|\sigma_1\|_{\Omega_*, \infty}$ with $M$ specified in Proposition \ref{prop_bound}.

By the equation for $v$, we have 
\begin{eqnarray*}
v(\cdot, t) &=&T_2(t)v_0+\int_0^t T_2(t-s)(B(\cdot, s)- \int_0^\infty(\lambda(a)-\lambda_*)i(\cdot, a, t)da)ds\\
          &\le & T_2(t)v_0+\int_0^t T_2(t-s)B(\cdot, s)dads.
\end{eqnarray*}
It then follows that 
\begin{eqnarray}
\|v(\cdot, t)\|_{\Omega, \infty}         &\le &  \|T_2(t)v_0\|_{\Omega, \infty}+\int_0^t \|T_2(t-s)B(\cdot, s)\|_{\Omega, \infty}dads    \nonumber\\
              &\le& M_1e^{-\lambda_* t}\|v_0\|_{\Omega, \infty} + \epsilon M_1 \|\sigma_2\|_{\Omega_*, \infty} \int_0^t e^{-\lambda_*(t-s)}\|\psi(\cdot, s)\|_{\Omega_*, \infty}ds.\label{conv2}
\end{eqnarray}
Combining \eqref{conv1} and \eqref{conv2}, we have 
\begin{equation*}
\|\psi(\cdot, t)\|_{\Omega_*, \infty}         \le  M_3\int_0^t e^{-\lambda_1(t-s)} e^{-\lambda_* s}ds+ \epsilon M_4 \int_0^t e^{-\lambda_1(t-s)} \int_0^s e^{-\lambda_*(s-r)}\|\psi(\cdot, r)\|_{\Omega_*, \infty}drds.
\end{equation*}
where $M_3=M_1M_2\|v_0\|_{\Omega, \infty}$ and $M_4=M_1M_2\|\sigma_2\|_{\Omega_*, \infty}$.
Notice that 
\begin{eqnarray*}
 \int_0^t e^{-\lambda_1(t-s)} \int_0^s e^{-\lambda_*(s-r)}\|\psi(\cdot, r)\|_{\Omega_*, \infty}drds&=&e^{-\lambda_1t}\int_0^t e^{\lambda_* r} \|\psi(\cdot, r)\|_{\Omega_*, \infty} \int_r^t e^{(\lambda_1-\lambda_*)s}dsdr\\
&\le& \frac{1}{\lambda_*-\lambda_1}\int_0^t e^{-\lambda_1(t-r)} \|\psi(\cdot, r)\|_{\Omega_*, \infty}dr.
\end{eqnarray*}
It then follows that 
\begin{equation*}
\|\psi(\cdot, t)\|_{\Omega_*, \infty}         \le  \frac{M_3}{\lambda_*-\lambda_1}e^{-\lambda_1 t}+ \frac{\epsilon M_4}{\lambda_*-\lambda_1}\int_0^t e^{-\lambda_1(t-r)} \|\psi(\cdot, r)\|_{\Omega_*, \infty}dr.
\end{equation*}
Choose $\epsilon=(\lambda_*-\lambda_1)\lambda_1/(2M_4)$. Then by the Gronwall's inequality, there exists $M_5>0$ such that for all $t\ge 0$
$$\|\psi(\cdot, t)\|_{\Omega_*, \infty}  \le M_5 e^{-\frac{\lambda_1 t}{2}}.$$
Therefore, \eqref{psi_uni} holds.  
\end{proof}

\begin{lemma}
Suppose that assumptions A1-A6 hold, and let $(\phi(x, t), \psi(x, t), u(x, t), i(x, a, t))$ be the solution of system \eqref{vector} and \eqref{host}. Then  $u_*>0$, where $u_*$ is specified in Theorem \ref{theorem_con}.
\end{lemma}

\begin{proof}
Assume to the contrary that $u_*=0$. By the first equation of \eqref{host}, we have
\begin{equation*}
\frac{\partial U}{\partial t}\ge - \|\sigma_2\|_{\Omega, \infty} \|\psi(\cdot, s)\|_{\Omega_*, \infty}U(t),
\end{equation*}
which implies that
\begin{equation*}
 U(t)\ge  U(0) \exp\left( -\|\sigma_2\|_{\Omega, \infty} \int_0^t \|\psi(\cdot, s)\|_{\Omega_*, \infty} ds\right) >0.
\end{equation*}
Hence $$u_*\ge \bar u_0\exp\left( -\|\sigma_2\|_{\Omega, \infty} \int_0^\infty \|\psi(\cdot, s)\|_{\Omega_*, \infty} ds\right) >0,$$
which is a contradiction  by Lemma \ref{lemma_psibound}. 
\end{proof}


\begin{thebibliography}{9}


\bibitem{Bailey1957} Bailey, N. T. J., The Mathematical Theory of Epidemics, Charles Griffin \& Company Limited, London, 1957.

\bibitem{Dietz1988} Dietz K., Mathematical models for transmission and control of malaria, {\em Malaria: principles and practice of malariology}, {\bf 2} (1988), 1091-1133.

\bibitem{Busenberg1988} Busenberg, S.  and Vargas, C., Modeling Chagas' disease: variable population size and demographic implications, {\em Mathematical Population Dynamics, Lecture Notes Pure and Applied Mathematics}, (1988), 283-295.

\bibitem{inaba2004mathematical} Inaba, H. and Sekine, H., A mathematical model for Chagas disease with infection-age-dependent infectivity, {\em Mathematical Biosciences}, {\bf 190(1)} (2004), 39-69.

\bibitem{Velasco1991} Velasco-Hern{\'a}ndez, J. X. H., An epidemiological model for the dynamics of Chagas' disease, {\em Biosystems}, {\bf 26(2)} (2004), 127-134.

\bibitem{Allen2008} Allen, L. J. S ., Bolker, B. M., Lou, Y. and Nevai, A. L., Asymptotic profiles of the steady states for an SIS epidemic reaction-diffusion model, {\em Discrete and Continuous Dynamical Systems}, {\bf 21(1)} (2008), 1-20.

\bibitem{Capasso1978} Capasso, V., Global solution for a diffusive nonlinear deterministic epidemic model, {\em SIAM Journal on Applied Mathematics}, {\bf 35(2)} (1978), 274-284.

\bibitem{Fitzgibbon1994} Fitzgibbon, W. E., Martin, C. B. and Morgan, J. J., A diffusive epidemic model with criss-cross dynamics, {\em Journal of Mathematical Analysis and Applications}, {\bf 184(3)} (1994), 399-414.

\bibitem{Fitzgibbon2008} Fitzgibbon, W. E. and Langlais, M., Simple models for the transmission of microparasites between host populations living on noncoincident spatial domains, Structured Population Models in Biology and Epidemiology, Springer , New York, 2008, 115-164.

\bibitem{Webb1981} Webb, G. F., A reaction-diffusion model for a deterministic diffusive epidemic, {\em Journal of Mathematical Analysis and Applications}, {\bf 84(1)} (1981), 150-161.

\bibitem{Knobler2006} Knobler, S., Mahmoud, A., Lemon, S. and Pray, L., The impact of globalization on infectious disease emergence and control: exploring the consequences and opportunities, workshop summary-forum on microbial threats, National Academies Press, 2006.

\bibitem{Webb1985theory} Webb, G. F., Theory of Nonlinear Age-dependent Population Dynamics, CRC Press, 1985.

\bibitem{Fitzgibbon1995} Fitzgibbon, W. E., Parrott, M. E. and Webb, G. F., Diffusion epidemic models with incubation and crisscross dynamics, {\em Mathematical Biosciences}, {\bf 128(1)} (1995), 131-155.

\bibitem{Fitzgibbon1996} Fitzgibbon, W. E., Parrott, M. E. and Webb, G. F.,  A diffusive age-structured SEIRS epidemic model, {\em Methods and Applications of Analysis}, {\bf 3} (1996), 358-369.

\bibitem{Langlais1988} Langlais, M.,  Large time behavior in a nonlinear age-dependent population dynamics problem with spatial diffusion, {\em Journal of Mathematical Biology}, {\bf 26(3)} (1988), 319-346.

\bibitem{Webb1980} Webb, G. F., An age-dependent epidemic model with spatial diffusion, {\em Archive for Rational Mechanics and Analysis}, {\bf 75(1)} (1980), 91-102.

\bibitem{Webb2008population} Webb, G. F., Population models structured by age, size, and spatial position, {\em Structured population models in biology and epidemiology}, (2008), 1-49.

\bibitem{Fitzgibbon2004} Fitzgibbon, W. E., Langlais, M. and Morgan, J. J.,  A reaction-diffusion system on noncoincident spatial domains modeling the circulation of a disease between two host populations, {\em Differential and integral equations}, {\bf 17(7-8)} (2004), 781-802.

\bibitem{Fitzgibbon2005} Fitzgibbon, W. E., Langlais, M., Marpeau, F. and Morgan, J. J.,  Modeling the circulation of a disease between two host populations on noncoincident spatial domains, {\em Biological Invasions}, {\bf 7(5)} (2005), 863-875.

\bibitem{Anita2009} Anita, S., Fitzgibbon, W. E. and Langlais, M.,  Global existence and internal stabilization for a reaction-diffusion system posed on noncoincident spatial domains, {\em Discrete and Continuous Dynamical Systems-Series B}, {\bf 11(4)} (2009), 805-822.

\bibitem{Cantrell2004} Cantrell, R. S. and Cosner, C., Spatial Ecology Via Reaction-diffusion Equations, John Wiley \& Sons, 1981.

\bibitem{Pazy2012} Pazy, A., Semigroups of Linear Operators and Applications to Partial Differential Equations, Springer Science \& Business Media, 2012.

\bibitem{Fitzgibbon19952} Fitzgibbon, W. E., Parrott, M. E. and Webb, G. F.,  Diffusive epidemic models with spatial and age dependent heterogeneity, {\em Discrete and Continuous Dynamical Systems}, {\bf 1} (2005), 35-57.


\end{thebibliography}
\end{document}